\documentclass{article}

\usepackage{cite}
\usepackage{amssymb}
\usepackage{amsmath}
\usepackage{amsfonts}
\usepackage{amsthm}
\usepackage{fancyhdr}
\usepackage[explicit]{titlesec}
\usepackage{titletoc}
\usepackage{booktabs}
\usepackage[all]{xy}
\usepackage[colorlinks=true,linkcolor=blue]{hyperref}
\usepackage{enumerate}
\usepackage{cancel}
\usepackage{enumitem}
\usepackage{colonequals}
\usepackage{accents}

\usepackage{url}

\usepackage{hyperref}

\newtheorem{theorem}{Theorem}[section]
\newtheorem{proposition}[theorem]{Proposition}
\newtheorem{problem}[theorem]{Problem}

\newtheorem*{theorem*}{Theorem}

\newtheorem{definition}[theorem]{Definition}

\numberwithin{equation}{section}

\newcommand{\dee}{\mathrm{d}}
\newcommand{\deee}{\hspace{2 pt} \mathrm{d}}

\allowdisplaybreaks

\begin{document}

\title{Weak equivalence of stationary actions and the entropy realization problem}

\author{Peter Burton, Martino Lupini and Omer Tamuz}

\date{\today}

\maketitle

\begin{abstract} We introduce the notion of weak containment for stationary actions of a countable group and define a natural topology on the space of weak equivalence classes. We prove that Furstenberg entropy is an invariant of weak equivalence, and moreover that it descends to a continuous function on the space of weak equivalence classes. \end{abstract}

\section{Introduction}

Let $G$ be a countable discrete group and let $m$ be a probability measure on $G$. Let also $(X,\mu)$ be a standard probability space. A measurable action of $G$ on $(X,\mu)$ is said to be $m$-stationary if the corresponding convolution of $m$ with $\mu$ is equal to $\mu$. More explicitly, this means $\sum_{g \in G} m(g) \cdot \mu(gA) = \mu(A)$ for all measurable subsets $A$ of $X$. Stationary actions are automatically nonsingular, and form a natural intermediate class between measure-preserving actions and general nonsingular actions. We will write $\mathrm{Stat}(G,m,X,\mu)$ for the set of $m$-stationary actions of $G$ on $(X,\mu)$. Given an action $a \in \mathrm{Stat}(G,m,X,\mu)$ we will write $g^a$ for the nonsingular transformation of $(X,\mu)$ corresponding to $g$. \\
\\
In \cite{K}, Kechris defined a notion of weak containment for measure-preserving actions of countable groups analogous to the standard notion of weak containment for unitary representations. The same definition can be given for stationary actions.

\begin{definition} Let $a,b \in \mathrm{Stat}(G,m,X,\mu)$. We say that $a$ is weakly contained in $b$, in symbols $a \preceq b$, if the following condition holds. For every $\epsilon > 0$, every finite $F \subseteq G$ and every finite collection $A_1,\ldots,A_n$ of measurable subsets of $X$, there are measurable subsets $B_1,\ldots,B_n$ of $X$ such that \[ |\mu(g^a A_i \cap A_j) - \mu(g^b B_i \cap B_j)| < \epsilon \] for all $g \in F$ and all $i,j \in \{1,\ldots,n\}$. We say that $a$ is weakly equivalent to $b$, in symbols $a \sim b$, if $a \preceq b$ and $b \preceq a$. \end{definition}

Thus $a$ is weakly contained in $b$ if the statistics of $a$ on finite partitions can be simulated arbitrary well in the action $b$. Weak equivalence is a much coarser relation than isomorphism; for example in \cite{FoWe04} it is shown that all free measure-preserving actions of an amenable group are weakly equivalent. It is also better behaved from the perspective of descriptive set theory: there is in general no standard Borel structure on the set of isomorphism classes of $m$-stationary actions, whereas in Section \ref{sec2} we will define a natural Polish topology on the set of weak equivalence classes of $m$-stationary actions for any pair $(G,m)$.\\
\\
In \cite{Fur63}, Furstenberg introduced an invariant $h_m(X,\mu,a)$ which quantifies how far an $m$-stationary action $a$ is from being measure-preserving. Later termed Furstenberg entropy, this is defined by \[ h_m(X,\mu,a) = - \sum_{g \in G} m(g) \cdot \int_X \log  \frac{\dee g^a \mu}{\dee \mu}(x)  \deee \mu(x).\]

By Jensen's inequality, we have that $h_m(X,\mu,a)$ is nonnegative, and it is zero if and only if $a$ is measure-preserving. The following problem has been studied in articles such as \cite{Bow14f}, \cite{BowHaTa16}, \cite{Dani16}, \cite{HarTa15}, \cite{KaiVer83} and \cite{Nev03}.

\begin{problem}[Furstenberg entropy realization problem] For a fixed pair $(G,m)$, describe the possible values of Furstenberg entropy on ergodic $\nu$-stationary systems. \label{prob1} \end{problem}

The goal of this note is to establish the following theorem, which shows that the above problem can be regarded as a problem about the structure of the space of weak equivalence classes.

\begin{theorem} \label{thm1} Furstenberg entropy is an invariant of weak equivalence and descends to a continuous function on the space of weak equivalence classes. \end{theorem}

\section{A characterization of weak containment}

In this section we verify that one obtains an equivalent notion if one alters the definition of weak containment to allow shifts on both sides of the intersections.

\begin{proposition}

Let $a,b \in \mathrm{Stat}(G,m,X,\mu)$. Then the following are equivalent.

\begin{enumerate}
\item[(i)] $a$ is weakly contained $b$.

\item[(ii)] For any finite subset $F$ of $G$, $\epsilon >0$, and measurable subsets $A_1,\ldots
,A_n$ of $X$, there exist measurable subsets $B_1,\ldots,B_n$ of $X$ such that \begin{equation} \label{eq7} \left\vert \mu(g^a A_i \cap h
^a A_j)-\mu(g^b B_i \cap h^b B_j) \right\vert <\epsilon \end{equation} for all $g,h\in F$ and $
i,j\in \left\{ 1,\ldots ,n\right\} $. \end{enumerate}

\end{proposition}

\begin{proof}
Taking $h = 1_G$ it is clear that ($\mathrm{ii}$) implies ($\mathrm{i}$). We now show ($\mathrm{i}$) implies ($\mathrm{ii}$). Suppose
that $F= \{ g_0,\ldots ,g_m \} $ is a finite subset of $G$, $
n$ is a natural number, and $A_0,\ldots ,A_n$ are measurable subsets of $X$. Without loss
of generality, we can assume that $n=m$, $g_0 = 1_G$ and $A_0 =X$. Fix $\epsilon >0$ and choose $0< \delta
< \epsilon /7$. Set $A_{i,j}=g^a_j A_{i}$ for $i,j  \in \{1,\ldots,n\}$. In particular we
have $A_{i,0}=A_i$ and $A_{i,j}=g^a_j A_{i,0}$ for $i,j \in \{1,\ldots,n\}$. By assumption
there exist measurable subsets $B_{i,j}$ of $X$ such that \[ | \mu (
A_{i,j}\cap g^a_m A_{l,k} ) -\mu ( B_{i,j} \cap g^b_m B_{l,k}) | <\delta \] for all $i,j,k,l,m \in \{1,\ldots,n\}$. Since $
A_{0,0}=X$ and $g_0 = 1_G$, we have that $\mu(B_{0,0}) >1-\delta $. It follows that \[ |\mu(g^a_m A_{l,k})-\mu(g_m^b B_{l,k})| < 2\delta \]  for $m,l,k \in \{1,\ldots,n\} $. Therefore

\begin{eqnarray*}
\mu(B_{j,m} \triangle g_m^b B_{j,0}) &=& \mu(B_{j,m}) + \mu(g_m^b B_{j,0})-2 \mu(B_{j,m} \cap g^b_m B_{j,0}) \\
&\leq & 6\delta +\mu(A_{j,m})+\mu_(g_m^a A_{j,0})-2 \mu (A_{j,m} \cap g_m^a A_{j,0}) \\
&=&6\delta +\mu(A_{j,m} \triangle g_m^a A_{j,0})=6\delta. \end{eqnarray*}%

In conclusion%
\begin{equation*} |\mu (g_k^a A_i \cap g_m^a A_j)-\mu(g_k^b B_{i,0} \cap g_m^b B_{j,0}) | 
\leq |\mu(g_k^a A_{i,0}\cap A_{j,m})-\mu(g_k^b B_{i,0} \cap B_{j,m})| +6\delta 
\leq 7\delta < \epsilon.
\end{equation*}%
for every $i,j,k,m \in \{1,\ldots,n\}$. Thus we can take $B_i = B_{i,0}$ to obtain (\ref{eq7}). \end{proof}

\section{The space of weak equivalence classes} \label{sec2}

For $a \in \mathrm{Stat}(G,m,X,\mu)$ we will write $\tilde{a}$ for the weak equivalence class of $a$. Let $(g_k)_{k=1}^\infty$ be an enumeration of $G$. For a natural number $m$ and an ordered finite partition $\overline{A} = \{A_1,\ldots,A_n\}$ of $X$, we will write $M_{m,\overline{A}}(a)$ for the point in $[0,1]^{m \times n \times n}$ whose $(k,i,j)$-coordinate is $\mu(g_k^a A_i \cap A_j)$. Let then $C_{m,n}(a)$ be the closure in $[0,1]^{m \times n \times n}$ of the set \[ \Bigl\{ M_{m,\overline{A}}: \overline{A} \mbox{ is a partition of }X\mbox{ into }n \mbox{ pieces.} \Bigr \}. \] Clearly we have $a \preceq b$ if and only if $C_{m,n}(a) \subseteq C_{m,n}(b)$ for all natural numbers $m,n$. Let \[ \delta(a,b) = \sum_{m,n = 1}^\infty \frac{1}{2^{m+n}} \cdot d_H \bigl ( C_{m,n}(a),C_{m,n}(b) \bigr), \] where $d_H$ is the Hausdorff distance on the space of compact subsets of $[0,1]^{m \times n \times n}$. Then for any $a,b,c,d \in \mathrm{Stat}(G,m,X,\mu)$ with $a \sim c$ and $b \sim d$ we have $\delta(a,b) = \delta(c,d)$. Thus the quantity $\tilde{\delta}(\tilde{a},\tilde{b}) \colonequals \delta(a,b)$ is a well-defined metric on the space of weak equivalence classes. The corresponding topology is easily seen to be Polish. We denote this space by $\widetilde{\mathrm{Stat}}(G,m,X,\mu)$. As in the measure-preserving case, an ultraproduct construction shows that $\widetilde{\mathrm{Stat}}(G,m,X,\mu)$ is compact.\\
\\
In addition to its topology, $\widetilde{\mathrm{Stat}}(G,m,X,\mu)$ carries a convex structure. Given $a,b \in \mathrm{Stat}(G,m,X,\mu)$, and $t \in (0,1)$ one can realize $a$ as an action on $[0,t)$ and realize $b$ as an action on $[t,1]$. One then defines $ta + (1-t)b$ to be the action on $[0,1]$ which agrees with $a$ on $[0,t)$ and $b$ on $[t,1]$. It is easy to see that this procedure gives a well-defined operation on $\widetilde{\mathrm{Stat}}(G,m,X,\mu)$. As in the measure-preserving case discussed in \cite{PBur1}, the convex structure is better behaved if one instead considers the relation $\preceq_s$ of stable weak containment. This is defined by letting $a \preceq_s b$ if and only if $a \preceq b \times \iota$, where $\iota$ is the trivial action of $G$ on a standard probability space. Write $\widetilde{\mathrm{Stat}}_s(G,m,X,\mu)$ for the space of stable weak equivalence classes. $\tilde{\delta}$ gives a Polish topology on $\widetilde{\mathrm{Stat}}_s(G,m,X,\mu)$ and since $h_m(X,\mu,a \times \iota) = h_m(X,\mu,a)$, Theorem \ref{thm1} continues to hold if we replace weak equivalence by stable weak equivalence. The arguments from \cite{PBur1} carry over to show that $\widetilde{\mathrm{Stat}}_s(G,m,X,\mu)$ is isomorphic to a compact convex subset of a Banach space, and that its extreme points are exactly those stable weak equivalence classes containing an ergodic action. Moreover, the map $a \mapsto h_m(X,\mu,a)$ respects the convex combination operation. Thus understanding the convex structure of $\widetilde{\mathrm{Stat}}_s(G,m,X,\mu)$ could give new understanding of Problem \ref{prob1}.

\section{Proof of Theorem \ref{thm1}}

For each $n$, let $a_n \in \mathrm{Stat}(G,m,X,\mu)$; let also $a \in \mathrm{Stat}(G,m,X,\mu)$. Assume that $\tilde{a}_n$ converges to $\tilde{a}$ in $\widetilde{\mathrm{Stat}}(G,m,X,\mu)$. Fixing $g \in G$, it is enough to show the following: for any $c \geq 0$ we have \[ \lim_{n \to \infty} \mu \left( \left\{x \in X: \frac{\dee g^{a_n} \mu}{\dee \mu}(x) > c \right \} \right) = \mu \left( \left\{x \in X: \frac{\dee g^a \mu}{\dee \mu}(x) > c \right \} \right). \]

Let $M$ be a positive constant such that $\frac{\dee g^b \mu}{\dee \mu} \leq M$ for any $m$-stationary action $a$. Let $\omega_n =\frac{\dee g^{a}\mu}{\dee \mu }$ and $\omega_n= \frac{\dee g^{a_n}\mu}{\dee \mu}$. Write $C= \{ x \in X:\omega(x) >c \} $, and $C_n= \left\{ x\in X:\omega _n\left( x\right) >c\right\} $. We will prove that $\mu \left( C\right) \leq \liminf_n\mu
\left( C_n\right) $. The proof that $\mu \left( C\right) \geq
\limsup_n\mu \left( C_n\right) $ is analogous. Suppose by contradiction $%
\mu \left( C\right) >\liminf_n\mu \left( C_n\right) $. Thus, after
passing to a subsequence, we can assume that there is $\delta >0$ such that $%
\mu \left( C_n\right) \leq \mu \left( C\right) -\delta $ for every $n\in 
\mathbb{N}$. Identify $X$ with $[0,1]$, so that we have a Borel linear order on $X$. Define the Borel linear order $\sqsubseteq $ on $X$ by letting $%
t\sqsubseteq s$ iff $\omega \left( t\right) <\omega \left( s\right) $ or $\omega
\left( t\right) =\omega \left( s\right) $ and $t<s$. Similarly define $\sqsubseteq
_n$ in terms of $\omega _n$. Note that if $D$ is a terminal segment of $%
\sqsubseteq $ then we have $\mu (g^{a}D)\geq \mu (g^{a}E)$ for any $E$ with $\mu
(E)=\mu (D)$. For $n\in \mathbb{N}$ write $D_n$ for the terminal segment
of $\sqsubseteq $ such that $\mu (D_n)=\mu _n(C_n)$ and write $E_n$ for
the terminal segment of $\sqsubseteq _n$ such that $\mu (C)=\mu _n(E_n)$.
Let also $F_n$ be the terminal segment of $\sqsubseteq $ such that $\mu
(F_n)=\mu (C_n)+\delta $ and let $K_n$ be the terminal segment of $%
\sqsubseteq _n$ such that $\mu _n(K_n)=\mu (C_n)+\delta $. Clearly $%
D_n\subseteq F_n\subseteq C$ and $C_n\subseteq K_n\subseteq E_n$.
We have 
\begin{equation}
\mu (F_n\setminus D_n)=\mu (F_n)-\mu (D_n)=\delta =\mu
_n(K_n)-\mu _n(C_n)=\mu _n(K_n\setminus C_n)  \label{eq14}
\end{equation}%
and similarly 
\begin{equation}
\mu (C\setminus F_n)=\mu _n(E_n\setminus K_n).  \label{eq1.1.1}
\end{equation}%
Note that since $\omega (x)>c\geq \omega _n(y)$ if $x\in C$ but $y\in
X\setminus C_n$, (\ref{eq1.1.1}) implies 
\begin{equation}
\mu \left( g^{a}(C\setminus F_n)\right) \geq \mu _n\left(
g^{a_n}(E_n\setminus K_n)\right) .  \label{eq15}
\end{equation}%
Let $H$ be the terminal segment of $\sqsubseteq $ such that $\mu (H)=\mu
(C)-\delta $ so that by (\ref{eq14}) we have $\delta =\mu (C\setminus H)=\mu
(F_n\setminus D_n)$. Since $F_n\setminus D_n\subseteq C$ and $%
C\setminus H$ has the lowest Radon-Nikodym derivative of any subset of $C$
with measure $\delta $ this implies 
\begin{equation}
\mu \left( g^{a}(C\setminus H)\right) \leq \mu \left( g^{a}(F_n\setminus
D_n)\right) .  \label{eq16}
\end{equation}%
For $n\in \mathbb{n}$ from (\ref{eq14}), (\ref{eq15}) and (\ref{eq16}) we
have 
\begin{align}
& \mu \left( g^{a}(C\setminus D_n)\right) -\mu \left(
g^{a_n}(E_n\setminus C_n)\right)  \notag \\
& =\mu \left( g^{a}(C\setminus F_n)\right) +\mu \left(
g^{a}(F_n\setminus D_n)\right) -\mu \left( g^{a_n}(E_n\setminus
K_n)\right) -\mu \left( g^{a_n}(K_n\setminus C_n)\right)  \notag \\
& \geq \mu \left( g^{a}(F_n\setminus D_n)\right) -\mu \left(
g^{a_n}(K_n\setminus C_n)\right) \geq \mu \left( g^{a}(F_n\setminus
D_n)\right) -c\cdot \mu \left( K_n\setminus C_n\right)  \notag \\
& =\mu \left( g^{a}(F_n\setminus D_n)\right) -c\delta \geq \mu \left(
g^{a}(C\setminus H)\right) -c\delta .  \label{eq17}
\end{align}%
For $x\in C$ we have $\omega (x)>c$ so the last quantity is strictly
positive. Choose 
\begin{equation}
0<\varepsilon <\frac{1}{2(4+M)}\cdot \left( \mu \left( g^{a}(C\setminus
H)\right) -c\delta \right) \text{.}  \label{eq10}
\end{equation}%
Since $\tilde{a}_n \to \tilde{a}$, for every Borel partition $A_{1},\ldots ,A_{k}$ of $X$
there is a partition $B_{1},\ldots ,B_{k}$ of $X$ such that $|\mu
(A_{i})-\mu (B_{i})|{}<\varepsilon $ and $\left\vert \mu (g^{a}A_{i}\cap
A_{j})-\mu (g^{a_n}B_{i}\cap B_{j})\right\vert <\varepsilon $ for all $%
i,j\in \{1,\ldots ,k\}$. Fixing $n$, write $C^{\prime }=C_n$, $D=D_n$, $%
E=E_n$ and $\sqsubseteq^{\prime }=\sqsubseteq_n$. Note that from (\ref{eq17}) and (\ref%
{eq10}) we have 
\begin{equation}
2(4+M)\varepsilon <\mu \left( g^{a}(C\setminus D)\right) -\mu \left(
g^{a_n}(E\setminus C^{\prime })\right) .  \label{eq12}
\end{equation}%
Let $A_{1}=X\setminus C$ and $A_{2}=C$. Find $B_{1},B_{2}\subseteq X$ such
that $|\mu (A_{i})-\mu (B_{i})|<\varepsilon $ and 
\[
\left\vert \mu (g^{a}A_{i}\cap A_{j})-\mu (g^{a_n}B_{i}\cap
B_{j})\right\vert <\varepsilon
\]
for each $i,j\in \{1,2\}.$ Note that 
\[
\mu \left( X\setminus (B_{1}\cup B_{2})\right) \leq 2\varepsilon .
\]
We have 
\begin{align}
\mu (g^{a}A_{1})& =\mu (g^{a}A_{1}\cap A_{1})+\mu (g^{a}A_{1}\cap A_{2}) 
\notag \\
& \geq \mu (g^{a_n}B_{1}\cap B_{1})+\mu (g^{a_n}B_{1}\cap
B_{2})-2\varepsilon  \notag \\
& \geq \mu (g^{a_n}B_{1}\cap B_{1})+\mu (g^{a_n}B_{1}\cap B_{2})+\mu
\left( g^{a_n}B_{1}\setminus (B_{1}\cup B_{2})\right) -4\varepsilon  \notag
\\
& \geq \mu (g^{a_n}B_{1})-4\varepsilon .  \label{eq1}
\end{align}%
Note that 
\[
\mu (B_{1})\geq \mu (A_{1})-\varepsilon =\mu (X\setminus C)-\varepsilon .
\]
Write $L$ for the initial segment of $\sqsubseteq^{\prime }$ such that $\mu (L)=\mu
(X\setminus C)-\varepsilon $. Note that $\mu (X\setminus E)=\mu (X\setminus
C)$ and so $\mu \left( X\setminus (E\cup L)\right) =\varepsilon $. We have 
\[
\mu \left( g^{a_n}(X\setminus E)\right) =\mu (g^{a_n}L)+\mu \left(
g^{a_n}\left( X\setminus (E\cup L)\right) \right)
\]
and therefore%
\begin{equation}
\mu (g^{a_n}L)\geq \mu \left( g^{a_n}(X\setminus E)\right) -M\varepsilon
.  \label{eq1.1.2}
\end{equation}%
Since $\mu (B_{1})\geq \mu (L)$ and $\mu (g^{a_n}L)\leq \mu (g^{a_n}J)$
for any $J\subseteq X$ with $\mu (J)\geq \mu (L)$ from (\ref{eq1.1.2}) we
see 
\[
\mu (g^{a_n}B_{1})\geq \mu \left( g^{a_n}(X\setminus E)\right)
-M\varepsilon .
\]
From (\ref{eq1}) we have 
\begin{equation}
\mu \left( g^{a}(X\setminus C)\right) \geq \mu \left( g^{a_n}(X\setminus
E)\right) -(4+M)\varepsilon .  \label{eq2}
\end{equation}%
Now write $A_{1}=C^{\prime }$ and $A_{2}=X\setminus C^{\prime }$. Find $%
B_{1},B_{2}\subseteq X$ such that 
\[
\left\vert \mu (A_{i}\cap A_{j})-\mu (B_{i}\cap B_{j})\right\vert
<\varepsilon
\]
and 
\[
\left\vert \mu (g^{a_n}A_{i}\cap A_{j})-\mu (g^{a}B_{i}\cap
B_{j})\right\vert <\varepsilon
\]
for each $i,j\in \{1,2\}$. Arguing as before we have $\mu
(g^{a_n}A_{1})\leq \mu (g^{a}B_{1})+4\varepsilon $ and $\mu
(g^{a}B_{1})\leq \mu (g^{a}D)+M\varepsilon $ so that 
\begin{equation}
\mu (g^{a_n}C^{\prime })\leq \mu (g^{a}D)+(4+M)\varepsilon .  \label{eq4}
\end{equation}%
From (\ref{eq2}) and (\ref{eq4}) we have%
\begin{equation}
\mu \left( g^{a}(X\setminus C)\cup D)\right) \geq \mu \left( g^{a_n}\left(
(X\setminus E)\cup C^{\prime }\right) \right) -2(4+M)\varepsilon .
\label{eq5}
\end{equation}%
Note that 
\[
D\sqcup (C\setminus D)\sqcup (X\setminus C)=X
\]
and 
\[
C^{\prime }\sqcup (E\setminus C^{\prime })\sqcup (X\setminus E)=X.
\]
Thus from (\ref{eq12}) and (\ref{eq5}) we have 
\begin{align*}
1& =\mu \left( g^{a}\left( D\cup (X\setminus C)\right) \right) +\mu \left(
g^{a}(C\setminus D)\right) \\
& \geq \mu \left( g^{a_n}\left( C^{\prime }\cup (X\setminus E)\right)
\right) -2(4+M)\varepsilon +\mu \left( g^{a}(C\setminus D)\right) \\
& >\mu \left( g^{a_n}\left( C^{\prime }\cup (X\setminus E)\right) \right)
+\mu \left( g^{a_n}(E\setminus C^{\prime })\right) =1
\end{align*}%
which is the desired contradiction. This concludes the proof of Theorem \ref{thm1}.

\bibliographystyle{plain}
\bibliography{bibliography2016.10.31}

Department of Mathematics\\
California Institute of Technology\\
Pasadena CA, 91125\\
\texttt{pjburton@caltech.edu}\\
\texttt{lupini@caltech.edu}\\
\texttt{omertamuz@gmail.edu}

\end{document}